\documentclass{amsart}
\usepackage{amsfonts}
\usepackage{latexsym}
\usepackage{amssymb}
\usepackage{amsmath}
\usepackage{bbm}
\usepackage{enumerate}

\allowdisplaybreaks

\newcommand{\R}{\mathbb R}
\newcommand{\N}{\mathbb N}

\newcommand{\E}{\mathbb E}
\newcommand{\Pro}{\mathbb P}
\newcommand{\dif}{\,\mathrm{d}}
\newcommand{\ext}{\mathrm{ext}}

\DeclareMathOperator*{\kmax}{k-max}

\DeclareMathOperator*{\onemax}{1-max}
\DeclareMathOperator*{\nmax}{n-max}

\DeclareMathOperator*{\tmax}{\lceil t\rceil-max}
\DeclareMathOperator*{\sconv}{s-conv}

\newtheorem{thm}{Theorem}[section]

\newtheorem{lemma}[thm]{Lemma}

\newtheorem{proposition}[thm]{Proposition}

\newtheorem{example}{Example}

\theoremstyle{remark}

\numberwithin{equation}{section}

\begin{document}


\title[On Averages of Order Statistics]{Uniform Estimates for Averages of Order Statistics of Matrices}

\author{Richard Lechner}
\address{Institute of Analysis, Johannes Kepler University Linz,
Altenberger Stra\ss e 69, 4040 Linz, Austria} \email{richard.lechner@jku.at}

\author{Markus Passenbrunner}
\address{Institute of Analysis, Johannes Kepler University Linz,
Altenberger Stra\ss e 69, 4040 Linz, Austria} \email{markus.passenbrunner@jku.at}

\author{Joscha Prochno}
\address{Institute of Analysis, Johannes Kepler University Linz,
Altenberger Stra\ss e 69, 4040 Linz, Austria} \email{joscha.prochno@jku.at}

\date{\today}

\begin{abstract}
We prove uniform estimates for the expected value of averages of order statistics of matrices in
terms of their largest entries. As an application, we obtain similar probabilistic estimates for $\ell_p$
norms via real interpolation.
\end{abstract}

\maketitle


\section{Introduction and Main Results} \label{sec:introduction}

Combinatorial and probabilistic inequalities play an important role in a variety of areas of
mathematics, especially in Banach space theory. In \cite{KS1} and \cite{KS2}, S.~Kwapie\'n and
C.~Sch\"utt studied combinatorial expressions involving matrices and obtained inequalities in terms
of the average of the largest entries of the matrix. To be more precise, they showed that
\begin{equation}\label{eq:Kwapien Schuett estimate max-norm}
\frac{1}{n!} \sum_{\pi \in\mathfrak{S}_n} \max_{1\leq i \leq n} |a_{i\pi(i)}| \simeq \frac{1}{n}\sum_{k=1}^n s(k),
\end{equation}
where $s(k)$ is the $k$-th largest entry of the matrix $a$ and $\mathfrak S_n$ the symmetric group.
This estimate seems crucial if one wants to compute the projection constant of symmetric Banach
spaces and related invariants. Among other things, the authors obtained estimates for the positive
projection constant of finite dimensional Orlicz spaces and estimated the order of the projection
constant of the Lorentz spaces $\ell_{2,1}^n$.
Also, the symmetric sublattices of $\ell_1(c_0)$ as well as the finite dimensional symmetric
subspaces of $\ell_1$ were characterized. Further applications and extensions of
\eqref{eq:Kwapien Schuett estimate max-norm} can be found in \cite{KS2,S1,S2,MSS,PS}, just to mention a
few.

The main result of this paper is a generalization of \eqref{eq:Kwapien Schuett estimate max-norm} in
the sense that we study the expected value of averages of higher order statistics of a matrix in a
more general setting described below.
Our method of proof is purely probabilistic in nature, whereas the proof of \eqref{eq:Kwapien Schuett estimate max-norm} in \cite{KS1} uses non-trivial combinatorial arguments.
 
In what follows, given a finite set $G$, we denote the normalized counting measure on $G$ by $\Pro$, i.e., 
\[
\Pro(E) = \frac{|E|}{|G|},\qquad E \subseteq G,
\]
where $|\cdot|$ denotes the cardinality. $\E$ will always denote the expectation with respect to the normalized counting measure.  Moreover, for a vector $x\in\R^n$ with non-negative entries, we denote its $k$-th largest entry by
\[
\kmax\limits_{1\leq i \leq n} x_i.
\]
In particular, $\onemax_{1\leq i \leq n} x_i$ is the maximal value, $\nmax_{1\leq i \leq n} x_i$ the minimal value of $x$. Our main result is the following:

\begin{thm}\label{thm:main}
Let $n,N\in\N$ and $a\in \R^{n\times N}$. Let $G$ be a collection of maps from $I=\{1,\dots,n\}$ to $J=\{1,\dots,N\}$ and $C_G>0$ be a constant only depending on $G$. Assume that for all $i\in I$, $j\in J$ and all different pairs $(i_1,j_1),(i_2,j_2)\in I\times J$ 
\begin{enumerate}[(i)]
\item\label{eq:condition 1} $\Pro(\{g\in G: g(i)=j\})=1/N$,
\item\label{eq:condition 2} $\Pro(\{g\in G: g(i_1)=j_1, g(i_2)=j_2 \}) \leq C_G/N^2$.
\end{enumerate}
Then, for any $\ell \leq n$, 
\begin{equation}\label{eq:main result firs part}
  \frac{c}{N} \sum_{j=1}^{\ell N}s(j) \leq \int_{G} \sum_{k=1}^{\ell}\kmax\limits_{1\leq i \leq n} |a_{ig(i)}| \dif \Pro(g) \leq \frac{2}{N} \sum_{j=1}^{\ell N}s(j),
\end{equation} 
where $c = 2^{-5}(1+2C_G)^{-2}$.
\end{thm}

Observe that estimate \eqref{eq:Kwapien Schuett estimate max-norm} \cite[Theorem 1.1]{KS1} is a special case of our result with the choice $\ell=1$ and $G=\mathfrak S_n$, and that for $\ell=1$ and $G=\{1,\dots,n\}^{\{1,\dots,n\}}$ we directly obtain \cite[Lemma 7]{GLSW1}. Note that in this general setting $\E \max_{1\leq i \leq n} |a_{ig(i)}|$ was already studied in \cite{KS2}. In a slightly different setting, order statistics were considered also in \cite{GLSW1,GLSW2,GLSW3,GLSW4,GLSW5}.

We will now present two natural choices for the set $G$ that appear frequently in the literature
(cf. \cite{KS1,KS2,S1,S2,S3,PS,GLSW1,G,JMST,P}).

\begin{example}
If $N=n$ and $G=\mathfrak S_n$ is the group of permutations of the numbers $\{1,\dots,n\}$, then
\[
\Pro(\pi(i)=j) = \frac{1}{n}, \qquad 1\leq i,j \leq n, 
\]
and for $(i_1,j_1)\neq (i_2,j_2)$
\[
\Pro(\pi(i_1)=j_1,\pi(i_2)=j_2) \leq  \frac{1}{n(n-1)} \leq \frac{2}{n^2}.
\] 
This means that $C_{G}\leq 2$. Hence, Theorem \ref{thm:main} implies
\[
\frac{1}{800} \frac{1}{n}\sum_{j=1}^{\ell n}s(j) \leq \frac{1}{n!} \sum_{\pi\in\mathfrak S_n} \sum_{k=1}^\ell \kmax_{1\leq i \leq n} |a_{i\pi(i)}| \leq 2 \frac{1}{n}\sum_{j=1}^{\ell n}s(j).
\]
\end{example}

\begin{example}
If $N=n$ and $G$ is the set of all mappings from $\{1,\dots,n\}$ into $\{1,\dots,n\}$, then
\[
\Pro(g(i)=j) = \frac{1}{n}, \qquad 1\leq i,j \leq n, 
\]
and for $(i_1,j_1)\neq (i_2,j_2)$
\[
\Pro(g(i_1)=j_1,g(i_2)=j_2) \le  \frac{1}{n^2}.
\] 
This means that $C_{G}=1$. Hence, Theorem \ref{thm:main} implies
\[
\frac{1}{288} \frac{1}{n}\sum_{j=1}^{\ell n}s(j) \leq \frac{1}{n^n} \sum_{g\in G} \sum_{k=1}^\ell \kmax_{1\leq i \leq n} |a_{ig(i)}| \leq 2 \frac{1}{n}\sum_{j=1}^{\ell n}s(j).
\]
\end{example}

Another combinatorial inequality that was obtained in \cite[Theorem 1.2]{KS1} and which turned out to be crucial to study and characterize symmetric subspaces of $L_1$ (cf. \cite{S1,S2,PS}) states that for all $1\leq p \leq \infty$
\begin{equation}\label{eq:Kwapien Schuett estimate p-norm}
\frac{1}{n!}\sum_{\pi\in\mathfrak S_n} \Big( \sum_{i=1}^n |a_{i\pi(i)}|^p \Big)^{1/p} \simeq \frac{1}{n}\sum_{k=1}^n s(k) + \Big(\frac{1}{n} \sum_{k=n+1}^{n^2} s(k)^p \Big)^{1/p}.
\end{equation}
In Section \ref{sec:applications}, we will use Theorem \ref{thm:main} to generalize this result and show that the lower bound in \eqref{eq:Kwapien Schuett estimate p-norm} can be naturally derived via real interpolation. The upper bound is quite easily obtained and we just follow \cite{KS1}. Please note that averages of order statistics of matrices naturally appear, as they are strongly related to the $K$-functional of the interpolation couple $(\ell_1,\ell_\infty)$. Again, two typical choices for the set of maps $G$ are $\mathfrak S_n$ and $\{1,\dots,n\}^{\{1,\dots,n\}}$. We will prove the following result:

\begin{thm}\label{thm:application}
Let $n,N\in\N$, $a\in \R^{n\times N}$, and $1\leq p < \infty$. Let $G$ be a collection of maps from $I=\{1,\dots,n\}$ to $J=\{1,\dots,N\}$ and $C_G>0$ be a constant only depending on $G$. Assume that for all $i\in I$, $j\in J$ and all different pairs $(i_1,j_1),(i_2,j_2)\in I\times J$ 
\begin{enumerate}[(i)]
\item $\Pro(\{g\in G: g(i)=j\})=1/N$,
\item $\Pro(\{g\in G: g(i_1)=j_1, g(i_2)=j_2 \}) \leq C_G/N^2$.
\end{enumerate}
Then
\begin{align*}
C\bigg[ \frac{1}{N} \sum_{k=1}^N s(k) + \Big(\frac{1}{N}\sum_{k=N+1}^{nN} s(k)^p\Big)^{1/p}\bigg] & \leq \E \Big( \sum_{i=1}^n |a_{ig(i)}|^p \Big)^{1/p} \\
& \leq \frac{1}{N} \sum_{k=1}^N s(k) + \Big(\frac{1}{N}\sum_{k=N+1}^{nN} s(k)^p\Big)^{1/p},
\end{align*}
where $C>0$ is a constant only depending on $C_G$.
\end{thm} 

The organization of the paper is as follows. In Section \ref{sec: lower bound}, we will prove the lower estimate in \eqref{eq:main result firs part}. This is done by reducing the problem to the case of matrices only taking values in $\{0,1\}$ and showing the estimate for this subclass of matrices. In Section \ref{subsec: upper bound G}, we establish the upper bound in \eqref{eq:main result firs part} by passing from averages of order statistics to equivalent Orlicz norms and using an extreme point argument. Section \ref{sec:applications} contains the proof of Theorem \ref{thm:application}.   

\section{Notation and Preliminaries}
 
Throughout this paper we will use $|E|$ to denote the cardinality  of a finite set $E$. By $\mathfrak S_n$ we denote the symmetric group on the set $\{1,\dots,n\}$. We will denote by $\lfloor x \rfloor$ and $\lceil x\rceil$ the largest integer $m\leq x$ and the smallest integer $m\geq x$, respectively.

For an arbitrary matrix $a=(a_{ij})_{i,j=1}^{n,N}$, we denote by $(s(k))_{k=1}^{nN}$ the decreasing rearrangement of $(|a_{ij}|)_{i,j=1}^{n,N}$. To avoid confusion, in certain cases we write $(s_a(k))_{k=1}^{nN}$ to emphasize the underlying matrix $a$. 

Please also recall that the Paley-Zygmund inequality for non-negative random variables $Z$ and $0<\theta<1$ states that
\begin{equation}\label{ine:paley zygmund}
 \Pro(Z\geq \theta \cdot\E Z)\geq (1-\theta)^2\frac{(\E Z)^2}{\E Z^2}.
\end{equation}

A convex function $M:[0,\infty)\rightarrow[0,\infty)$ is called an Orlicz function if $M(0)=0$ and if $M$ is not constant. Given an Orlicz function $M$, the Orlicz sequence space $\ell_M^n$ is $\R^n$ equipped with the Luxemburg norm 
$$
  \|x\|_M=\inf\bigg\{\lambda >0 : \sum_{i=1}^n M\left(\frac{|x_i|}{\lambda}\right)\leq 1\bigg\}.
$$
For example, the classical $\ell_p$ spaces are Orlicz spaces with $M(t)=p^{-1}t^p$. The closed unit ball of the space $\ell_M^n$ will be denoted by $B_M^n$. We write $\ext (B_M^n)$ for the set of extreme points of $B_M^n$ and $\sconv (M)$ shall denote the set of points of strict convexity of $M$. 
We will make use of the following characterization of extreme points of $B_M^n$:
\begin{lemma}[\cite{W}, Lemma 1]\label{lem:extreme points of orlicz balls}
Let $M$ be an Orlicz function. Then $x\in \ext (B_M^n)$ if and only if
\begin{enumerate}[(i)]
\item $\sum_{i=1}^nM(|x_i|)=1$,
\item there exists at most one index $i_0\in\N$, $1\leq i_0 \leq n$ such that $x_{i_0}\not\in \pm\sconv (M)$. 
\end{enumerate}
\end{lemma}

For a detailed and thorough introduction to Orlicz spaces we refer the reader to \cite{RR} or \cite{LT77}.

\section{The lower bound} \label{sec: lower bound}

In this section we will prove the lower bound in \eqref{eq:main result firs part}. We begin by
recalling some notation and assumptions given in Theorem \ref{thm:main}.
Let $a\in\R^{n\times N}$, $I=\{1,\ldots,n\}$, $J=\{1,\ldots,N\}$, and $G$ be a collection of
maps from $I$ to $J$. The matrix $a$ will be fixed throughout the entire section.
By $\Pro$ we denote the normalized counting measure on $G$, i.e., $\Pro(E) = |E|/|G|$ for $E\subset G$. We assume a uniform distribution of the random variable $g \mapsto g(i)$ for each $i\in I$, i.e.,
\begin{equation*}
\Pro(g(i)=j)=\frac{1}{N},\qquad i\in I,j\in J.
\end{equation*}
We assume for all different pairs $(i_1,j_1),(i_2,j_2)\in I\times J$ that
\begin{equation*} 
\Pro(g(i_1)=j_1, g(i_2)=j_2 ) \leq \frac{C_G}{N^2},
\end{equation*}
with a constant $C_G\geq 1$ that depends on $G$, but not on $n$ or $N$.  

Without loss of generality, we will assume that $a$ has only non-negative entries. It is enough to show the lower estimate in \eqref{eq:main result firs part} for matrices $a$ that consist of only the $\ell N$ largest entries, while all others are equal to zero. This is because if we change any entry $a_{i_0j_0}\leq s(\ell N+1)$ by setting $a_{i_0j_0}=0$, the left hand side in \eqref{eq:main result firs part} remains the same, while $\kmax\limits_{1\leq i \leq n} |a_{ig(i)}|$ does not increase for any $g\in G$. 

\subsection{The key ingredients}

We will now introduce a bijective function $h$ that determines the ordering of the values of $a$. The crucial point is that this function does not depend on the actual values of the matrix, but merely on their relative size. So let $h:\{1,\dots,n\cdot N\}\rightarrow I\times J$ be a bijective function satisfying
\begin{equation}\label{eq:condition on b}
\begin{aligned}
a(h(j)) & \geq a(h(j+1)),& 1 &\leq j \leq \ell N,\\
a(h(j))& =0, & \ell N +1& \leq j \leq nN.
\end{aligned}
\end{equation}
Observe that there is possibly more than one choice for $h$, since some of the entries of the matrix $a$ might have the same value.  

For all $j\in\N$, $1\leq j \leq n\cdot N$, define the random variable
 	\[
 	Y_j:G\to \{0,1\},\qquad
 	Y_j(g)= \begin{cases}
 	1, & \text{if }h(j)\in g, \\
 	0, & \text{if }h(j)\notin g,
 	\end{cases}
 	\]
and given $m\in\N$, $1\leq m \leq n\cdot N$, let
\[
X_m:G \to \{0,1,\dots, n\}, \qquad
 X_m(g):= \sum_{j=1}^mY_j(g)=|h(\{1,\dots,m\})\cap g|,
\]
where we identify $g$ with its graph $\{(i,g(i)):i\in I \}$. $X_m$ counts the number of elements in the path $\{(i,g(i)):i\in I \}$ that intersect with the positions of the $m$ largest entries of $a$. As we will see in Subsection \ref{subsec:reduction}, the random variables $X_m$ are strongly related to order statistics.

In Lemma \ref{lem:keq1}, Lemma \ref{lem:Xm}, and Lemma \ref{lem:XmXln}, we investigate crucial properties of the distribution function of $X_m$.  

\begin{lemma}\label{lem:keq1}
	For all $m\in\N$, $1\leq m \leq n\cdot N$, we have
	\begin{equation}\label{eq:Pinduction}
	\Pro(X_m\geq 1)\geq\frac{m}{N}\Big(1-C_G\frac{m-1}{2N}\Big).
	\end{equation}
	In particular, 
	\[
	\Pro(X_{\lceil N/C_G \rceil}\geq 1)\geq 
	\frac{1}{2C_G}.
	\]
\end{lemma}
\begin{proof}
By using the inclusion-exclusion principle, we obtain
\begin{align*}
\Pro(X_m \geq 1) & = \Pro \Big( \bigcup_{j=1}^m \{g\in G: Y_j(g)=1\}\Big) \\
& \geq \sum_{j=1}^m \Pro(Y_j=1) - \sum_{i<j}\Pro(Y_i=1,Y_j=1) \\
& \geq \frac{m}{N} - \frac{m(m-1)C_G}{2N^2} \\
& = \frac{m}{N} \Big( 1- C_G\frac{m-1}{2N}\Big),
\end{align*}
where the latter inequality is a direct consequence of conditions \eqref{eq:condition 1} and \eqref{eq:condition 2} in Theorem \ref{thm:main}. 
\end{proof}

\begin{lemma} \label{lem:Xm}
	For all $m\in\N$, $1\leq m \leq n\cdot N$, and all $\theta\in(0,1)$, we have
	\begin{equation}\label{eq:paley estimate}
	 \Pro\Big(X_m\geq \theta\cdot \frac{m}{N}\Big) \geq (1-\theta)^2\frac{m}{N+m\cdot C_G}.
	\end{equation}
	
\end{lemma}
\begin{proof}
 	The result follows as a consequence of Paley-Zygmund's inequality (cf. \eqref{ine:paley zygmund}). Therefore, we need to compute $\E X_m$ and $\E X_m^2$. Note that $\E Y_j=\Pro(Y_j=1)=1/N$ and thus $\E X_m=\sum_{j=1}^m Y_j=m/N$.
 	Moreover, since $Y_j=Y_j^2$, we have
 	\begin{align*}
 	\E X_m^2 &=\sum_{i,j=1}^m \E Y_i Y_j = \sum_{j=1}^m \E Y_j + \sum_{i\neq j} \E Y_i Y_j \\
 	& \leq \frac{m}{N}+ C_G\frac{m(m-1)}{N^2},
 	\end{align*}
 	where the latter inequality is a direct consequence of conditions \eqref{eq:condition 1} and \eqref{eq:condition 2} in Theorem \ref{thm:main}.
 	Inserting those estimates in \eqref{ine:paley zygmund}, we obtain the result. 
\end{proof}

\begin{lemma}\label{lem:XmXln}
	For all $m\in\N$, $1\leq m \leq n\cdot N$, we have
	\[
		\Pro(X_m\geq 1)\geq \min\left\{\frac{m}{2N},\frac{1}{2C_G}\right\}\Pro(X_{\ell N}\geq 1),
	\]
	and for all $k,m\in\N$ with $2kN\leq m\leq n N$
	\begin{equation}\label{eq:Xmk}
		\Pro(X_m\geq k)\geq \frac{\Pro(X_{\ell N}\geq k)}{2+4 C_G}.
 	\end{equation}
\end{lemma}
\begin{proof}
	Let $1\leq m\leq n\cdot N$. If $m\leq N/C_G$, Lemma \ref{lem:keq1} implies 
	\[
	  \Pro(X_m\geq 1)\geq \frac{m}{2N}\geq \frac{m}{2N}\cdot\Pro(X_{\ell N}\geq 1)=\min\left\{\frac{m}{2N},\frac{1}{2C_G}\right\}\Pro(X_{\ell N}\geq 1).
	\]
	On the other hand, if $m\geq N/C_G$, Lemma \ref{lem:keq1} implies
	\begin{align*}
	  \Pro(X_m\geq 1) &\geq \Pro(X_{\lceil N/C_G\rceil}\geq 1)
	   \geq \frac{1}{2C_G} \\
	  & \geq \frac{1}{2C_G}\Pro(X_{\ell N}\geq 1)  = \min\left\{\frac{m}{2N},\frac{1}{2C_G}\right\}\Pro(X_{\ell N}\geq 1).
	\end{align*}
	Now we prove \eqref{eq:Xmk}. Let $k\leq n/2$ and $m$ such that $2kN\leq m\leq n\cdot N$. Then Lemma~\ref{lem:Xm} with $\theta=1/2$ implies
	\[
	  \Pro(X_m\geq k)\geq \Pro(X_{2kN}\geq k)\geq \frac{1}{2+4C_G}\geq \frac{1}{2+4C_G}\Pro(X_{\ell N}\geq k). \qedhere
	\]
\end{proof}

\subsection{Reduction to two valued matrices}\label{subsec:reduction}

We will now reduce the problem of estimating the expected value of averages of order statistics of general matrices to matrices only taking one value different from zero. To do so, we need some more definitions.  

Let $\mathcal A_h$ be the collection of all non-negative real $n\times N$ matrices $b$ that satisfy 
\begin{equation}
\begin{aligned}
b(h(j)) & \geq b(h(j+1)),& 1 &\leq j \leq \ell N,\\
b(h(j))& =0, & \ell N +1& \leq j \leq nN.
\end{aligned}
\end{equation}
For every $b\in \mathcal A_h$, we set
\[
\widetilde{b}(h(j)):=\Big(\frac{1}{\ell N}\sum_{i=1}^{\ell N} b(h(i))\Big) \cdot \mathbbm 1_{\{1,\dots,\ell N\}} (j),\qquad 1\leq j\leq n\cdot N,
\]
as the matrix that contains the averaged entries of $b$. Note that $\widetilde b \in\mathcal A_h$. Moreover, we define 
\[
a_m(h(k)):=\mathbbm 1_{\{1,\dots, m\}}(k), \qquad 1\leq m \leq n\cdot N.
\]
Observe that $a_m\in \mathcal A_h$ for all $1\leq m \leq n\cdot N$. For $b\in\mathcal A_h$ and $g\in G$ we put
\[
S_k(b)(g):= \kmax_{1\leq i \leq n} b_{ig(i)}\qquad\text{and}\qquad S(b)(g):=\sum_{k=1}^\ell S_k(b)(g).
\]
	
\begin{lemma}\label{thm:am}
Let $m\in\N$, $1\leq m \leq \ell N$. Then we have
\[
\E S(\widetilde{a_m})\leq (8+16C_G)\cdot \E S(a_m).
\]
\end{lemma}
\begin{proof}
Observe that for every integer $k$ with $1\leq k\leq \ell$,
\begin{equation*}
\begin{aligned}
\E S_k(a_m)&=\Pro(S_k(a_m)=1)=\Pro(X_m\geq k),\\
\E S_k(\widetilde{a_m})&=\frac{m}{\ell N}\cdot\Pro\Big(S_k(\widetilde{ a_m})=\frac{m}{\ell N}\Big)= \frac{m}{\ell N}\cdot\Pro(X_{\ell N}\geq k).
\end{aligned}
\end{equation*}
As a consequence,
\begin{equation*}
\begin{aligned}
\E S(\widetilde{a_m})=\frac{m}{\ell N}\sum_{k=1}^\ell \Pro(X_{\ell N}\geq k)
\qquad\text{and}\qquad  \E S(a_m)=\sum_{k=1}^\ell \Pro(X_m\geq k).
\end{aligned}
\end{equation*}
Thus, in order to prove the lemma, it is enough to show that
\[
\frac{m}{\ell N}\sum_{k=1}^\ell \Pro(X_{\ell N}\geq k) \leq (8+16C_G) \cdot \sum_{k=1}^\ell \Pro(X_m\geq k)
\]
for $1\leq m\leq \ell N$.
First, we assume $m\leq 2N$. Then, Lemma \ref{lem:XmXln} implies 
\begin{align*}
\frac{m}{\ell N}\sum_{k=1}^{\ell}\Pro(X_{\ell N}\geq k) 
& \leq \frac{m}{N}\cdot \Pro(X_{\ell N}\geq 1) \\
& \leq \frac{2}{N}\max\{N,m\cdot C_G\} \Pro(X_m\geq 1) \\
&\leq 4C_G\cdot \sum_{k=1}^\ell \Pro(X_m\geq k)
\end{align*}
i.e., the assertion of the lemma for $m\leq 2N$.

Now, let $m\geq 2N+1$ and choose the integer $t\geq 1$ such that $2tN+1\leq m\leq 2(t+1)N$. The sequence $k\mapsto\Pro(X_{\ell N} \geq k)$ is decreasing, hence, noting that $t\leq \ell$,
\[
\frac{m}{\ell N}\sum_{k=1}^\ell \Pro(X_{\ell N}\geq k)\leq \frac{m}{tN}\sum_{k=1}^t \Pro(X_{\ell N}\geq k). 
\]
Then, estimate \eqref{eq:Xmk} of Lemma \ref{lem:XmXln} implies
\begin{align*}
\frac{m}{\ell N}\sum_{k=1}^\ell \Pro(X_{\ell N}\geq k)
& \leq \frac{m}{tN}\sum_{k=1}^t  (2+4C_G)\cdot\Pro(X_m\geq k) \\
&\leq \frac{(4+8C_G)(t+1)}{t}\sum_{k=1}^t \Pro(X_m\geq k) \\
&\leq (8+16C_G)\cdot \sum_{k=1}^\ell \Pro(X_m\geq k)
\end{align*}
and the result follows.
\end{proof}

\begin{lemma}\label{thm: estimate a and tilde a general case}
 We have
 \[
\E S(\widetilde a)\leq (8+16 C_G)\cdot \E S(a).
\] 
\end{lemma}
\begin{proof}
Recall that $X_j(g) = |h(\{1,\dots,j\})\cap g|$. Hence, for all $b\in\mathcal A_h$, 
\begin{align*}
\E S(b) & = \sum_{k=1}^{\ell}\sum_{j=1}^{\ell N} b(h(j))\cdot \Pro(\{g:X_{j-1}(g)=k-1,h(j)\in g \}).
\end{align*}
Defining
\[
f(j):= \sum_{k=1}^{\ell} \Pro(\{g:X_{j-1}(g)=k-1,h(j)\in g\}), \qquad 1\leq j \leq \ell N,
\]
we can write
\[
\E S(b) = \sum_{j=1}^{\ell N}f(j)b(h(j)).
\]
Since $a,\widetilde a\in\mathcal A_h$, $a(h(j))=s_a(j)$ and $\widetilde a(h(j)) = (\ell N)^{-1}\sum_{i=1}^{\ell N}s_a(i)$ for all $j \leq \ell N$, we obtain
\begin{equation*}\label{eq:expectation S with f and g}
\E S(a)=\sum_{j=1}^{\ell N} f(j) s_a(j) \qquad \text{and}\qquad \E S(\widetilde{a})= \sum_{j=1}^{\ell N } \widetilde f(j)s_a(j),
\end{equation*}
where for all $1\leq j \leq \ell N$
\[
\widetilde f(j) = \frac{1}{\ell N} \sum_{i=1}^{\ell N} f(i).
\]
Note that the functions $f$ and $\widetilde f$ only depend on $h$, i.e., only on the positions of the entries in the matrix and not on their values.
Since $a_m(h(j)) = 1$ for $j \leq m$ (zero otherwise) and $a_m,\widetilde{a_m}\in\mathcal A_h$, we have
\[
\E S(a_m)=\sum_{j=1}^m f(j) \qquad\text{and}\qquad \E S(\widetilde{a_m})=\sum_{j=1}^m \widetilde f(j).
\]
Now we conclude with $C=8+16C_G$ that
\begin{align*}
  C\E S(a)- \E S(\widetilde a)&= s_a(1)[Cf(1)-\widetilde f(1)]+\sum_{j=2}^{\ell N} [Cf(j)-\widetilde f(j)] s_a(j) \\
  &\geq s_a(2) \sum_{j=1}^2 [C f(j)-\widetilde f(j)]+ \sum_{j=3}^{\ell N}[Cf(j)-\widetilde f(j)] s_a(j) 
\end{align*}
where we used Lemma \ref{thm:am} for $m=1$.
Continuing in this fashion and using Lemma \ref{thm:am} for $m=2,\ldots,\ell N$, we obtain
\begin{equation*}
  C\E S(a)- \E S(\widetilde a) \geq 0.
  \qedhere
\end{equation*}
\end{proof}

\subsection{Conclusion}\label{subsec:conclusion}

As we have seen, we can reduce the case of general $a$ to multiples of matrices only taking values zero and one. Before we finally prove the lower bound in the main theorem, we will need another simple lemma.

\begin{lemma}\label{prop:lower estimate k-max}
	Let $b\in \mathcal A_h$ be an $(n\times N)$-matrix consisting of $\ell N$ ones and $(n-\ell)N$ zeros. Then, for all $1\leq k\leq \ell/2$, 
	$$
	\E \kmax_{1\leq i \leq n} b_{ig(i)} \geq \frac{1}{2+4C_G}. 
	$$
\end{lemma}
\begin{proof}
	Let $k\leq \ell/2$. Using Lemma \ref{lem:Xm} with $\theta=1/2$, we obtain 
	\begin{align*}
	\E \kmax_{1\leq i \leq n} b_{i g(i)} 
	&\geq \int_{\{g: X_{2kN}(g)\geq k\}} \kmax_{1\leq i \leq n} b_{ig(i)} \dif\Pro(g)\\
	& = \Pro(X_{2kN}\geq k)
	\geq \frac{k}{2(1+2kC_G)}\\
	& \geq \frac{1}{2+4C_G}. \qedhere
	\end{align*}
\end{proof}

\subsection*{Proof of the lower estimate in Theorem \ref{thm:main}}

By Theorem \ref{thm: estimate a and tilde a general case} we obtain 
\begin{align*}
\E \sum_{k=1}^\ell \kmax_{1\leq i \leq n} a_{ig(i)} & \geq \frac{1}{8(1+2C_G)} \E \sum_{k=1}^\ell \kmax_{1\leq i \leq n} \widetilde a_{ig(i)}.
\end{align*}
Now take $b\in\mathcal A_h$ consisting of $\ell N$ ones and $(n-\ell)N$ zeros such that 
\[
\left(\frac{1}{\ell N}\sum_{i=1}^{\ell N}s_a(i)\right) \cdot b=\widetilde a. 
\]
Then, by Lemma \ref{prop:lower estimate k-max}
\begin{align*}
 \E \sum_{k=1}^\ell \kmax_{1\leq i \leq n} \widetilde a_{ig(i)} 
& = \left[\E \sum_{k=1}^\ell \kmax_{1\leq i \leq n} b_{ig(i)}\right]  \frac{1}{\ell N} \sum_{i=1}^{\ell N}s_a(i) \\
& \geq \left[\E \sum_{k=1}^{\ell/2} \kmax_{1\leq i \leq n} b_{ig(i)}\right]  \frac{1}{\ell N} \sum_{i=1}^{\ell N}s_a(i) \\
& \geq \frac{1}{4+8C_G}\frac{1}{N} \sum_{j=1}^{\ell N}s(j).
\end{align*}
Combining the above estimates yields
\begin{equation}\label{eq:lower_estimate}
  \E \sum_{k=1}^\ell \kmax_{1\leq i \leq n} a_{ig(i)}
  \geq \frac{1}{32(1+2C_G)^2} \frac{1}{N} \sum_{j=1}^{\ell N}s(j),
\end{equation}
which is the lower estimate in Theorem \ref{thm:main}.

\section{The upper bound}\label{subsec: upper bound G}

We will now prove the upper bound of Theorem \ref{thm:main} via an extreme point argument. To do so, we first use the fact that the average of the $j\leq n\cdot N$ largest entries of a matrix $a\in\R^{n\times N}$ is equivalent to an Orlicz norm $\|a\|_{M_j}$ (cf. Lemma \ref{lem:approximation by orlicz norm}). Then, since the expected value of the average of order statistics defines a norm on $\R^{n \times N}$ as well, it is enough to prove the upper bound in Theorem \ref{thm:main} for the extreme points of $B^n_{M_j}$.  

Recall that, for a vector $(x_i)_{i=1}^n\in\R^n$, we denote  the decreasing rearrangement of $(|x_i|)_{i=1}^n$ by $(x_i^*)_{i=1}^n$. 
We start with the approximation of sums of decreasing rearrangements of vectors $x\in\R^n$ by
equivalent  Orlicz norms.

The following result is due to C. Sch\"utt (private communication).
With his permission we include it here.

\begin{lemma}\label{lem:approximation by orlicz norm}
Let $j\in\N$, $1\leq j \leq n$. Then, for all $x\in\R^n$, we have
\[
\frac{1}{2} \sum_{i=1}^j x_i^* \leq \|x\|_{M_j} \leq \sum_{i=1}^j x_i^*,
\]
where
\begin{equation}\label{eq:definition of Mk}
M_j(t) :=
\begin{cases} 
   			     0, & \quad 0\leq t \leq 1/j,\\
                 t-1/j , & \quad 1/j<t. 
 \end{cases}
\end{equation}
\end{lemma}
\begin{proof}
Let $x\in\R^n$. We start with the right hand side inequality. 
Of course,
\[
\frac{1}{j}\sum_{i=1}^jx_i^* \geq x_k^*,\qquad\forall j \leq k \leq n.
\]
Hence, for all $k\geq j$,
\begin{equation*}
M_j\left(\frac{x_k^*}{\sum_{i=1}^jx_i^*} \right) \leq M_j\left(\frac{1}{j} \right) = 0.
\end{equation*}
Therefore, we obtain
\[
\sum_{k=1}^n M_j\left( \frac{|x_k|}{\sum_{i=1}^jx_i^*}\right)
= \sum_{k=1}^{j-1} M_j\left( \frac{x_k^*}{\sum_{i=1}^jx_i^*}\right)
\leq \sum_{i=1}^{j-1} \frac{x_k^*}{\sum_{i=1}^jx_i^*} \leq 1.
\]

Now the other inequality. Take $\gamma< 1/2$. Then, since $M_j(t) \geq t-1/j$ for all $t\geq 0$, we have
\begin{align*}
\sum_{k=1}^n M_j\left(\frac{|x_k|}{\gamma \sum_{i=1}^jx_i^*} \right)
& \geq \sum_{k=1}^j M_j\left(\frac{x_k^*}{\gamma \sum_{i=1}^jx_i^*} \right) \\
& \geq \sum_{k=1}^j \left( \frac{x_k^*}{\gamma \sum_{j=1}^ks(j)} - \frac{1}{k} \right) \\
& = \frac{1}{\gamma}-1 >1.
\end{align*}
Therefore, we have for all $\alpha<1/2$
\[
\|x\|_{M_j} \geq \alpha \sum_{i=1}^jx_i^*. \qedhere
\]
\end{proof}

We are now able to prove the upper bound of Theorem \ref{thm:main}.

\begin{proposition}\label{prop:upper bound average}
Let $a\in\R^{n\times N}$. Then, for all $\ell \leq n$,
\begin{equation}\label{eq:estimate Mk norm}
\E \sum_{k=1}^\ell \kmax_{1\leq i \leq n} a_{ig(i)} \leq  \frac{2}{N} \|a\|_{M_{\ell N}}.
\end{equation}
\end{proposition}
\begin{proof}
It is sufficient to show \eqref{eq:estimate Mk norm} for all $a\in\ext{(B^{nN}_{M_\ell})}$. Therefore, by Lemma \ref{lem:extreme points of orlicz balls} (2), we only need to consider matrices $a\in\R^{n \times N}$ that are of the form
\begin{equation}\label{eq:extreme points of B_Mk}
a_{ij} :=
\begin{cases}
                 \frac{1}{\ell N}, & (i,j)\neq (i_0,j_0),\\
                 1+\frac{1}{\ell N} , & (i,j)= (i_0,j_0)
\end{cases}
\end{equation}
for some index pair $(i_0,j_0)\in I\times J$
or satisfy $a_{ij}=\frac{1}{\ell N}$ for all $i=1,\dots,n$, $j=1,\dots, N$. However, the latter choice of $a$ does not satisfy condition (1) in Lemma \ref{lem:extreme points of orlicz balls}, since in that case $\sum_{i=1}^{nN}M_{\ell N}(s_a(i))=0$. So the extreme points of $B_{M_{\ell N}}^{nN}$ with positive entries are given by \eqref{eq:extreme points of B_Mk}. Now, let $a$ be such a point in $B^{nN}_{M_{\ell N}}$. Then  
\begin{align*}
& \E \sum_{k=1}^\ell \kmax_{1\leq i \leq n} a_{ig(i)} \\
& = \int_{\{g:g(i_0)=j_0 \}} \sum_{k=1}^\ell \kmax_{1\leq i \leq n} a_{ig(i)} \dif \Pro(g) +  \int_{\{g:g(i_0)\neq j_0 \}} \sum_{k=1}^\ell \kmax_{1\leq i \leq n} a_{ig(i)} \dif \Pro(g) \\
& = \frac{2}{N}.
\end{align*}
On the other hand, we also have
\[
\frac{1}{N}\sum_{j=1}^{\ell N} s(j) = \frac{2}{N}.
\]
Therefore,
\[
\E \sum_{k=1}^\ell \kmax_{1\leq i \leq n} a_{ig(i)} = \frac{1}{N}\sum_{j=1}^{\ell N} s(j),
\]
Since by Lemma \ref{lem:approximation by orlicz norm}
\[
\sum_{j=1}^{\ell N} s(j) \leq 2 \|a\|_{M_{\ell N}},
\]
the result follows.
\end{proof}

\subsection*{Conclusion of the proof of Theorem~\ref{thm:main}}

Combining Lemma~\ref{lem:approximation by orlicz norm} and
Proposition~\ref{prop:upper bound average}, we get
\begin{equation}\label{eq:upper_estimate}
  \E \sum_{k=1}^\ell \kmax_{1\leq i \leq n} a_{ig(i)} \leq  \frac{2}{N} \sum_{i=1}^{\ell N}s(i),
\end{equation}
which is the upper estimate in Theorem \ref{thm:main}.
Inequalities~\eqref{eq:lower_estimate} and~\eqref{eq:upper_estimate} together complete the proof.

\section{An application of Theorem~\ref{thm:main}}\label{sec:applications}

We now present an application and use Theorem \ref{thm:main} to prove Theorem \ref{thm:application}. The proof uses real interpolation and is, what we find, a natural approach to combinatorial inequalities such as \eqref{eq:Kwapien Schuett estimate p-norm} that were obtained in \cite{KS1}. Please notice that \cite[Theorem 1.2]{KS1} is a special case of Theorem \ref{thm:application} when $G=\mathfrak S_n$. 

Let us first recall some basic notions from interpolation theory. A pair $(X_0,X_1)$ of Banach spaces is called a compatible couple if there is some Hausdorff topological space $\mathcal H$, in which each of $X_0$ and $X_1$ is continuously embedded. For example, $(L_1,L_\infty)$ is a compatible couple, since $L_1$ and $L_\infty$ are continuously embedded into the space of measurable functions that are finite almost everywhere. Of course, any pair $(X,Y)$ for which one of the spaces is continuously embedded in the other is a compatible couple.

For a compatible couple $(X_0,X_1)$ (with corresponding Hausdorff space $\mathcal H$), we equip $X_0+X_1$ with the norm
\begin{equation}\label{eq:norm on sum of spaces}
\| x \|_{X_0+X_1} := \inf_{x=x_0+x_1}\big(\|x_0\|_{X_0} + \|x_1\|_{X_1}\big),
\end{equation}
under which this space becomes a Banach space. This definition is independent of the particular space $\mathcal H$.
 
The $K$-functional is constructed from the expression \eqref{eq:norm on sum of spaces} by introducing a positive weighting factor $t>0$, as follows:

Let $(X_0,X_1)$ be a compatible couple. The $K$-functional is defined for each $f\in X_0+X_1$ and $t>0$ by
\[
K(f,t) = K(f,t;X_0,X_1) := \inf_{f=f_0+f_1}(\|f_0\|_{X_0} + t \|f_1\|_{X_1}),
\]
where the infimum extends over all representations $f=f_0+f_1$ of $f$ with $f_0 \in X_0$ and $f_1\in X_1$. 

Now, let $(X_0,X_1)$ be a compatible couple and suppose $0<\theta < 1$, $1\leq q < \infty$ or $0\leq \theta \leq 1$ and $q=\infty$. The space $(X_0,X_1)_{\theta,q}$ consists of all $f\in X_0+X_1$ for which the functional
\begin{equation*}\label{eqn:norm on theta q spaces}
  \|f\|_{\theta,q} := 
    \begin{cases}
                 \left( \int_{0}^\infty \big[ t^{-\theta}K(f,t;X_0,X_1)\big]^q \frac{\dif t}{t}\right)^{1/q}, & \quad 0<\theta<1,~1\leq q < \infty  ,\\
                 \sup_{t>0}t^{-\theta}K(f,t;X_0,X_1) , & \quad 0 \leq \theta \leq 1, ~q=\infty,
    \end{cases}
 \end{equation*}
is finite.

\subsection*{Proof of Theorem \ref{thm:application}}
To show the upper bound we use the same argument as in \cite{KS1}. For the sake of completeness we include it here. Let $a\in\R^{n\times N}$ and write $a=a'+a''$, where $a'$ contains the $N$ largest entries of $a$ and zeros elsewhere, and $a''$ contains $s(N+1)\dots,s(nN)$ and zeros elsewhere. Then, using triangle and Jensen's inequality, we obtain
\begin{align*}
\E \left(\sum_{i=1}^n |a_{ig(i)}|^p\right)^{1/p} & \leq \E \left(\sum_{i=1}^n |a'_{ig(i)}|^p\right)^{1/p} +  \E \left(\sum_{i=1}^n |a''_{ig(i)}|^p\right)^{1/p} \\
& \leq \E \sum_{i=1}^n |a'_{ig(i)}| +  \E\left(\sum_{i=1}^n |a''_{ig(i)}|^p\right)^{1/p} \\
& \leq \sum_{k=1}^{N}\sum_{i=1}^n \frac{1}{N}a_{ik}' + \left( \sum_{k=1}^N \sum_{i=1}^n\frac{1}{N}(a_{ik}'')^p \right)^{1/p} \\
& = \frac{1}{N} \sum_{k=1}^N s(k) + \left(\frac{1}{N}\sum_{k=N+1}^{nN}s(k)^p \right)^{1/p}.
\end{align*}

We will now prove the lower bound. Let $1\leq p < \infty$ and $\theta=1-1/p$.
First, recall that
\[
\|a\|_{\theta,p} = \left( \int_{0}^\infty
  \Big[t^{-\theta}K\big(a,t;L_1^{|G|}(\ell_1^n),L_1^{|G|}(\ell_\infty^n)\big)\Big]^p
  \frac{\dif t}{t}
\right)^{1/p}.
\]
Second, observe that
\begin{align*}
  K\big(a,t;L_1^{|G|}(\ell_1^n),L_1^{|G|}(\ell_\infty^n)\big)
  & = \inf_{a=b+c} \|b\|_{L_1^{|G|}(\ell_1^n)} + t\, \|c\|_{L_1^{|G|}(\ell_\infty^n)}\\
  & = \inf_{a=b+c} \int_G \|b(g)\|_{\ell_1^n} + t\, \|c(g)\|_{\ell_\infty^n} \dif\Pro(g)\\
  & = \int_G \inf_{a(g)=b(g)+c(g)} \|b(g)\|_{\ell_1^n} + t\, \|c(g)\|_{\ell_\infty^n} \dif\Pro(g).
\end{align*}
Hence, we have
\begin{equation}\label{eq:representation k functional}
  K\big(a,t;L_1^{|G|}(\ell_1^n),L_1^{|G|}(\ell_\infty^n)\big)
  = \int_G K\big(a(g),t;\ell_1^n,\ell_\infty^n\big)\dif\Pro(g).
\end{equation}
Third, the triangle inequality for integrals yields
\begin{align*}
  \|a\|_{\theta,p}
  & = \bigg( \int_0^\infty [
    t^{-\theta} K\big(a,t;L_1^{|G|}(\ell_1^n),L_1^{|G|}(\ell_\infty^n)\big)
  ]^p
  \frac{\dif t}{t} \bigg)^{1/p}\\
  & = \bigg( \int_0^\infty [
    \int_G t^{-\theta} K\big(a(g),t;\ell_1^n,\ell_\infty^n)\big) \dif\Pro(g)
  ]^p
  \frac{\dif t}{t} \bigg)^{1/p}\\
  & \leq \int_G \bigg( \int_0^\infty
    \Big[t^{-\theta} K\big(a(g),t;\ell_1^n,\ell_\infty^n)\big)\Big]^p
   \frac{\dif t}{t} \bigg)^{1/p}
  \dif{\Pro(g)}\\
  & = \int_G \|a(g)\|_{\ell_p^n} \dif{\Pro(g)}.
\end{align*}
Therefore, we have
\[
\|a\|_{L_1^{|G|}(\ell_p^n)} \geq \|a\|_{\theta,p}, 
\]
for all $a:G\to\R^n$, $a(g)(i)=a_{ig(i)}$. Now we compute the $K$-functional of $(L_1^{|G|}(\ell_1^n),L_1^{|G|}(\ell_\infty^n))$ using \eqref{eq:representation k functional}. First, observe that 
\begin{align*}
K(a(g),t;\ell_1^n,\ell_\infty^n) & = \int_0^{\lfloor t\rfloor}(a(g))^*(s) \dif s + \int_{\lfloor t\rfloor}^t(a(g))^*(s) \dif s \\
& = \sum_{k=1}^{\lfloor t\rfloor}\kmax |a(g)| + (t-\lfloor t\rfloor)\cdot \tmax |a(g)|.
\end{align*}
Then, using \eqref{eq:representation k functional}, we obtain 
\begin{align*}
  \|a\|_{1-\frac{1}{p},p}^p & = \int_{0}^\infty t^{-p}K(a,t;L_1^{|G|}(\ell_1^n),L_1^{|G|}(\ell_\infty^n))^p \dif t \\
& = \int_{0}^\infty t^{-p}\Big( \int_G K(a(g),t;\ell_1^n,\ell_\infty^n) \dif\Pro(g) \Big)^p \dif t \\
& = \int_{0}^\infty t^{-p}\Big( \int_G \sum_{k=1}^{\lfloor t\rfloor}\kmax |a(g)| + (t-\lfloor t\rfloor)\cdot \tmax |a(g)|) \dif\Pro(g) \Big)^p \dif t \\
& \geq \int_0^1 \big(\E a(g)^*(1) \big)^p \dif t + \int_{1}^{n+1} t^{-p} \Big(\E \sum_{k=1}^{\lfloor t\rfloor} a(g)^*(k) \Big)^p \dif t \\
& \geq c_1 \bigg[ \big(\E a(g)^*(1) \big)^p + \sum_{\ell=1}^{n}  \Big(\E \frac{1}{\ell}\sum_{k=1}^{\ell} a(g)^*(k) \Big)^p\bigg],
\end{align*}
where $c_1$ is a positive absolute constant. By Theorem \ref{thm:main}, we get
\begin{align*}
	\Big[\E a(g)^*(1) \Big]^p + \sum_{\ell=1}^{n}  \Big(\E \frac{1}{\ell}\sum_{k=1}^{\ell} a(g)^*(k) \Big)^p 
	& \geq c_2\bigg[ \Big(\frac{1}{N}\sum_{j=1}^Ns(j) \Big)^p + \sum_{\ell=1}^n \Big(\frac{1}{\ell N}\sum_{j=1}^{\ell N} s(j) \Big)^p \bigg]\\
	& \geq c_2\bigg[ \Big(\frac{1}{N}\sum_{j=1}^Ns(j) \Big)^p + \sum_{\ell=1}^n \frac{1}{N} \sum_{j=\ell N+1}^{(\ell+1)N}s(j)^p \bigg] \\
	& = c_2 \bigg[\Big(\frac{1}{N}\sum_{j=1}^Ns(j) \Big)^p +  \frac{1}{N} \sum_{\ell=N+1}^{nN} s(j)^p\bigg],
\end{align*}
where $c_2$ is a positive constant only depending on $C_G$.
Taking the $p$-th root concludes the proof.


\subsection*{Acknowledgments}
We would like to thank our colleague Erhard Aichinger for helpful discussions.

The first named author is supported by the Austrian Science Fund, FWF P23987 and FWF P22549.
The second named author is supported by the Austrian Science Fund, FWF P23987.
The third named author is supported by the Austrian Science Fund, FWFM 1628000.

\bibliographystyle{abbrv}
\bibliography{combinatorial_inequalities_averages_order_statistics}

\end{document}